\documentclass[11pt]{article}
\usepackage{amsmath,amsthm}
\usepackage{amssymb,mathrsfs}
\usepackage{bm,mathtools}

\usepackage{xcolor}
\usepackage{geometry}
\usepackage[colorlinks=true,
linkcolor=blue,citecolor=blue,
urlcolor=blue,pagebackref]{hyperref}

\makeatletter
\def\@seccntDot{.}
\def\@seccntformat#1{\csname the#1\endcsname\@seccntDot\hskip 0.5em}
\renewcommand\section{\@startsection{section}{1}{\z@}%
{18\p@ \@plus 6\p@ \@minus 3\p@}%
{9\p@ \@plus 6\p@ \@minus 3\p@}%
{\large\bfseries\boldmath}}
\renewcommand\subsection{\@startsection{subsection}{2}{\z@}%
{12\p@ \@plus 6\p@ \@minus 3\p@}%
{3\p@ \@plus 6\p@ \@minus 3\p@}%
{\bfseries\boldmath}}
\renewcommand\subsubsection{\@startsection{subsubsection}{3}{\z@}%
{12\p@ \@plus 6\p@ \@minus 3\p@}%
{\p@}%
{\bfseries\boldmath}}
\makeatother

\usepackage{microtype}

\theoremstyle{plain}
\newtheorem{theorem}{Theorem}[section]
\newtheorem{lemma}{Lemma}[section]

\newtheorem{corollary}{Corollary}[section]
\newtheorem{proposition}{Proposition}[section]

\theoremstyle{definition}

\numberwithin{equation}{section}
\allowdisplaybreaks
\title{The largest Laplacian eigenvalue of induced-$K_{1,r}$-free graphs}

\author{Lele Liu\footnote{School of Mathematical Sciences, Anhui University, Hefei 230601, 
P.R. China. E-mail: \texttt{liu@ahu.edu.cn} (L. Liu). Supported by the National
Nature Science Foundation of China (No. 12471320) and the Anhui Provincial Natural Science Foundation (No. 2408085Y003).}
~~ and ~~
Bo Ning\footnote{Corresponding author. College of Computer Science, Nankai University, Tianjin 300350, P.R. China.
E-mail: \texttt{bo.ning@nankai.edu.cn} (B. Ning). Partially supported by the National Nature Science
Foundation of China (No. 12371350).}}

\date{}

\begin{document}
\maketitle

\begin{abstract}
Let $G$ be a simple graph of maximum degree $d$, and let
$\mu(G)$ denote the largest eigenvalue of its Laplacian matrix. 
For a fixed integer $k\geq 2$, Aharoni,
Alon, and Berger (2016) asked whether every graph containing no induced copy of
$K_{1,k}$ satisfies
$\mu(G)\leq (2 - \frac{2}{k} + o(1)) d$.
We answer this question by proving the stronger sharp bound
\[
\mu(G)\leq \left(2-\frac{2}{k}\right)(d+1).
\]
% More generally, if every open neighborhood of $G$ has independence number at
% most $r$, then $\mu(G)\leq 2r(d+1)/(r+1)$. Equality is attained by clique
% extensions of $r$-regular bipartite graphs. 
The proof combines a sign decomposition of a Laplacian Rayleigh vector with a 
weighted local Caro-Wei type inequality for independent sets.
\par\vspace{2mm}

\par\vspace{2mm}

\end{abstract}

\section{Introduction}

Throughout the paper, all graphs are finite and simple. For a graph $G$, let $A(G)$
and $D(G)$ denote its adjacency matrix and diagonal degree matrix, respectively.
The Laplacian matrix is $L(G)=D(G)-A(G)$, and its largest eigenvalue is denoted 
by $\mu(G)$. We write $N_G(v)$ for the open neighborhood of a vertex $v$ and 
$\Delta(G)$ for the maximum degree. The general estimate
$\mu(G)\leq 2\Delta(G)$
is classical; if $G$ is connected, equality holds precisely when $G$ is
regular and bipartite \cite{AndersonMorley}. Standard background on graph
spectra and Laplacian eigenvalues can be found in \cite{BrouwerHaemers}.

For an integer $k\geq 2$, a graph is  induced $K_{1,k}$-free if it contains no induced
copy of $K_{1,k}$. Equivalently,
$\alpha (G[N_G(v)])\leq k-1$ for every $v\in V(G)$.
Aharoni, Alon, and Berger \cite{AharoniAlonBerger} initiated the study of the
largest Laplacian eigenvalue under this local restriction. Let $t(d,k)$ be
the minimum number of edges in a graph on $d$ vertices with independence
number at most $k-1$. They proved
$\mu(G)\leq 2d-\frac{t(d,k)}{d-1}$
for an induced $K_{1,k}$-free graph of maximum degree $d>1$. The complement of
Tur\'{a}n's theorem \cite{Turan} gives $t(d,k) = (1 + o(1)) \frac{d^2}{2(k-1)}$,
and hence their result yields
\[
\mu(G)\leq \left(2-\frac{1}{2k-2}+o(1)\right) d.
\]
They also constructed induced $K_{1,k}$-free graphs with Laplacian spectral radius
asymptotic to $(2-2/k)d$ and asked whether this smaller coefficient is always
an upper bound up to a lower-order term. The regular case was subsequently
revisited, in terms of the least adjacency eigenvalue, by Cioab\u{a}, Elzinga,
and Gregory \cite{CioabaElzingaGregory}. It was also mentioned in the survey \cite{KoolenCaoYang2011}. 
To the best of our knowledge, this question has remained open since it was 
posed by Aharoni, Alon, and Berger in 2016.

We settle the question in the following form. It is convenient to state
the result using the local independence number.

\begin{theorem}\label{thm:main}
Let $k\geq 2$ be an integer, and let $G$ be a graph such that
$\alpha (G[N_G(v)])\leq k-1$ for every $v\in V(G)$. Then
\[
\mu(G)\leq \left(2-\frac{2}{k}\right) (\Delta(G) + 1).
\]
\end{theorem}

For regular graphs, Theorem~\ref{thm:main} is equivalently a lower bound for
the least adjacency eigenvalue.

\begin{corollary}\label{cor:least}
Let $k\geq 2$. If $G$ is a $d$-regular $K_{1,k}$-free graph and $\lambda_{\min}(G)$ is the
least eigenvalue of $A(G)$, then
\[
\lambda_{\min}(G)\geq -\frac{k-2}{k} d - \frac{2(k-1)}{k}.
\]
\end{corollary}

The proof of Theorem~\ref{thm:main} is given in Section~\ref{sec:main}.  
Its main ingredient is a weighted form of the local-minimum proof of the
Caro--Wei bound \cite{Caro,Wei}, and a transformation of the local 
independence condition into a quadratic form. The full details appear 
in Section~\ref{sec:main}. We then discuss applications and sharpness in Sections~\ref{sec:applications} and~\ref{sec:sharpness}, respectively

\section{Proof of our main results}
\label{sec:main}

We first record the weighted independent-set inequality used in the proof, which is a corollary of Theorem 1.4 in \cite{Bansal-Harris2023}. We include the proof here for the sake of completeness.

\begin{lemma}\label{lem:weighted-CW}
Let $G$ be a graph, and assign a positive real number $x_v$ to every $v\in V(G)$. Then
\[
\sum_{v\in V(G)} \frac{x_v}{x_v + \sum_{u\in N_G(v)} x_u}
\leq \alpha(G).
\]
\end{lemma}

\begin{proof}
For every $v\in V(G)$, let $X_v$ be an exponential random variable of rate
$x_v$, and take these random variables independently. Let
\[
I = \{v\in V(G): X_v < X_u\text{ for every }u\in N_G(v)\}.
\]
By construction, $I$ is an independent set. Moreover,
\[
\mathrm{Pr}(v\in I)
= \int_0^\infty x_v \mathrm{e}^{-x_v t}
\prod_{u\in N_G(v)} \mathrm{e}^{-x_u t}\,\mathrm{d}t 
= \frac{x_v}{x_v + \sum_{u\in N_G(v)} x_u}.
\]
It follows that
\[
\sum_{v\in V(G)}
\frac{x_v}{x_v + \sum_{u\in N_G(v)} x_u}
=\mathrm{E}|I|
\leq \alpha(G),
\]
as required.
\end{proof}

The next lemma converts the local independence condition into a quadratic-form inequality. 

\begin{lemma}\label{lem:matrix}
Let $G$ and $H$ be graphs on a common vertex set $U$, and suppose that
$\alpha (H[N_G(v)])\leq r$ for every $v\in U$. Then every vector 
$\bm{y}\in\mathbb{R}^U$ with positive coordinates satisfies
\[
\bm{y}^{\mathrm{T}} A(G) \bm{y}
\leq r\big(\bm{y}^{\mathrm{T}} \bm{y} + \bm{y}^{\mathrm{T}} A(H) \bm{y}\big).
\]
\end{lemma}

\begin{proof}
For $u\in U$, set $z_u:= 1 + \frac{(H\bm{y})_u}{y_u}$,
and let $Z$ be the diagonal matrix with diagonal entries $z_u$.  Fix
$v\in U$. Applying Lemma~\ref{lem:weighted-CW} to $H[N_G(v)]$, with vertex
weights $y_u$, gives
\begin{align}
\sum_{u\in N_G(v)} \frac{1}{z_u}
& = \sum_{u\in N_G(v)} \frac{y_u}{y_u+(A(H) \bm{y})_u} \nonumber \\
& \leq \sum_{u\in N_G(v)} \frac{y_u}{y_u + \sum_{w\in N_G(v)\cap N_H(u)} y_w} \nonumber \\
& \leq r. \label{eq:rows}
\end{align}
Thus every row sum of the nonnegative matrix $A(G) Z^{-1}$ is at most $r$.
Consequently, every eigenvalue $\theta$ of $A(G)Z^{-1}$ satisfies
$|\theta|\leq r$.  

The matrix $A(G)Z^{-1}$ is similar to the symmetric matrix
$M:= Z^{-1/2}A(G)Z^{-1/2}$, because $M=Z^{-1/2}(A(G)Z^{-1})Z^{1/2}$.  
Hence the largest eigenvalue of $M$ is at most $r$. Applying the Rayleigh 
inequality for $M$ to $Z^{1/2} \bm{y}$ yields
\begin{align*}
\bm{y}^{\mathrm{T}} A(G)\bm{y}
& = (Z^{1/2} \bm{y})^{\mathrm{T}} M(Z^{1/2}\bm{y})\leq r\,\bm{y}^{\mathrm{T}} Z \bm{y} \\
& = r\bigg(\bm{y}^{\mathrm{T}} \bm{y} + \sum_{u\in U} y_u (H\bm{y})_u\bigg) \\
& = r\big(\bm{y}^{\mathrm{T}} \bm{y} + \bm{y}^{\mathrm{T}} H\bm{y}\big),
\end{align*}
as desired.
\end{proof}

\begin{proof}[Proof of Theorem~\ref{thm:main}]
Let $\bm{x}\in\mathbb{R}^{V(G)}$ be nonzero, and put
\[
U:= \{v\in V(G): x_v\neq 0\}, \quad y_v:= |x_v| \quad (v\in U).
\]
Define two spanning subgraphs $F$ and $H$ of $G[U]$ by
\[
\begin{aligned}
    E(F) & =\{uv\in E(G[U]):x_ux_v<0\}, \\
    E(H) & =\{uv\in E(G[U]):x_ux_v>0\}.
\end{aligned}
\]
Thus $F$ consists of the edges whose endpoints have opposite signs, whereas
$H$ consists of the edges whose endpoints have the same sign. For every
$v\in U$, all vertices in $N_F(v)$ have the same sign. It follows that $H[N_F(v)] = G[N_F(v)]$.

Set for short
\[
\beta:= \frac{\bm{y}^{\mathrm{T}} A(F) \bm{y}}{\bm{y}^{\mathrm{T}} \bm{y}},
\qquad
\eta:= \frac{\bm{y}^{\mathrm{T}} A(H) \bm{y}}{\bm{y}^{\mathrm{T}} \bm{y}}, \qquad d:= \Delta(G).
\]
Every adjacency term with an endpoint outside $U$ is zero. Consequently,
\begin{align}
\bm{x}^{\mathrm{T}} L(G) \bm{x}
& = \sum_{v\in U} d_G(v) y_v^2 + \bm{y}^{\mathrm{T}} A(F) \bm{y} - \bm{y}^{\mathrm{T}} A(H) \bm{y} \nonumber \\
& \leq (d + \beta - \eta) \bm{y}^{\mathrm{T}} \bm{y}. \label{eq:rayleigh-first}
\end{align}
The inequality $2y_uy_v\leq y_u^2+y_v^2$, summed over the edges of $G[U]$, gives
\[
\bm{y}^{\mathrm{T}} (A(F) + A(H)) \bm{y}
= 2\sum_{uv\in E(G[U])} y_uy_v\leq \sum_{v\in U} d_{G[U]}(v) y_v^2\leq d\,\bm{y}^{\mathrm{T}} \bm{y}.
\]
Hence $\beta + \eta\leq d$. By Lemma \ref{lem:matrix},
$\beta\leq (k-1)(1 + \eta)$, or equivalently, $\beta - (k-1)\eta\leq k-1$. 
Since $k\geq 2$, the coefficients in the identity
\[
\beta - \eta = \frac{2}{k}(\beta - (k-1)\eta) + \frac{k-2}{k}(\beta + \eta)
\]
are nonnegative. Combining $\beta + \eta\leq d$ and
$\beta - (k-1)\eta\leq k-1$, we obtain
\[
\beta - \eta \leq \frac{2(k-1) + (k-2)d}{k}.
\]
Substitution into \eqref{eq:rayleigh-first} yields
\[
\frac{\bm{x}^{\mathrm{T}} L(G) \bm{x}}{\bm{x}^{\mathrm{T}} \bm{x}}
\leq d + \frac{2(k-1) + (k-2)d}{k} = \frac{2(k-1)}{k}(d+1).
\]
Taking the maximum over all nonzero vectors $\bm{x}$ proves the theorem.
\end{proof}

\section{Applications to independence complexes and independent transversals}
\label{sec:applications}

The original motivation for bounding the largest Laplacian eigenvalue of an induced
$K_{1,k}$-free graph came in part from the topology of its independence
complex and from independent-transversal problems. We record the consequences
of Theorem~\ref{thm:main} in these settings.

For a graph $G$, let $\mathcal{I}(G)$ denote its independence complex. For a
simplicial complex $\mathcal{C}$, let $\eta_{\mathrm{H}}(\mathcal{C})$ denote its
homological connectivity; thus $\eta_{\mathrm{H}}(\mathcal{ C})$ is the largest
integer $q$ such that
$\widetilde{H}_j(\mathcal{C}; \mathbb{R}) = 0$ for every $j\leq q-2$, where $\widetilde{H}_j(\mathcal{C}; \mathbb{R})$ is the $j$-th reduced homology group of the simplicial complex $\mathcal{C}$.
As usual, $\eta_{\mathrm H}(\mathcal C)=\infty$ if all the reduced homology
groups of $\mathcal C$ vanish.  A theorem of Aharoni, Berger, and
Meshulam~\cite{AharoniBergerMeshulam} gives the spectral bound
\begin{equation}\label{eq:homological-spectral-bound}
\eta_{\mathrm{H}}(\mathcal{I}(G))
\geq \frac{|V(G)|}{\mu(G)}
\end{equation}
whenever $G$ has at least one edge. If $G$ is edgeless, then
$\mathcal I(G)$ is a simplex and its homological connectivity is infinite.

\begin{corollary}\label{cor:homological-connectivity}
Let $k\geq 2$ be an integer, and let $G$ be a graph on $n$ vertices such 
that $G$ is induced $K_{1,k}$-free. If $d=\Delta(G)$, then
\[
\eta_{\mathrm{H}}(\mathcal{I}(G))
\geq \left\lceil\frac{kn}{2(k-1)(d+1)}\right\rceil.
\]
\end{corollary}

\begin{proof}
If $G$ is edgeless, then $\mathcal{I}(G)$ is a simplex and the result is
immediate. Otherwise, combining \eqref{eq:homological-spectral-bound} with
Theorem~\ref{thm:main} gives the desired result. 
\end{proof}

We next turn to independent transversals. Let
$\mathcal P=\{V_1,\ldots,V_m\}$ be a partition of $V(G)$. An independent
transversal of $\mathcal P$ is an independent set containing exactly one
vertex from each part.  We use the homological form of the topological Hall
criterion of Meshulam~\cite{MeshulamClique}; see also
\cite[Section~3]{AharoniAlonBerger}. It states that if
\begin{equation}\label{eq:topological-hall}
\eta_{\mathrm{H}}\left(\mathcal{I}\left(G\left[\cup_{i\in S} V_i\right]\right)\right) \geq |S|
\end{equation}
for every $S\subseteq [m]$, then $\mathcal{P}$ has an independent transversal.

\begin{corollary}\label{cor:independent-transversal}
Let $r\ge1$ and $d\ge0$ be integers, and let $G$ be a graph of maximum degree
at most $d$ such that $\alpha\big(G[N_G(v)]\big)\leq r$ for every $v\in V(G)$.
Suppose that $V(G)$ is partitioned into sets $V_1,\ldots,V_m$ satisfying 
$|V_i|\geq\big\lceil\frac{2r}{r+1}(d+1)\big\rceil$ $(1\leq i\leq m)$. 
Then the partition has an independent transversal.
Consequently, if $G$ is $K_{1,k}$-free, it is sufficient that
$|V_i|\geq\left\lceil\left(2-\frac{2}{k}\right)(d+1)\right\rceil$ $(1\leq i\leq m)$.
If $d=0$, the conclusion holds under the weaker assumption $|V_i|\ge1$ for
every $i$.
\end{corollary}

\begin{proof}
Put $h:= \big\lceil\frac{2r}{r+1}(d+1)\big\rceil$.
Fix a nonempty set $S\subseteq[m]$, and let
$F=G\left[\cup_{i\in S} V_i\right]$.
The local independence condition is inherited by induced subgraphs, and
$\Delta(F)\leq d$. Moreover,
$|V(F)| = \sum_{i\in S} |V_i|\geq |S|h$.
If $F$ is edgeless, then $\mathcal{I}(F)$ is a simplex, so
\eqref{eq:topological-hall} holds. Otherwise, Theorem~\ref{thm:main} and the
definition of $h$ give
\[
\mu(F)\leq \frac{2r}{r+1}\big(\Delta (F) + 1\big)
\leq \frac{2r}{r+1}(d+1) \leq h.
\]
Using \eqref{eq:homological-spectral-bound}, we obtain
\[
\eta_{\mathrm{H}}(\mathcal{I}(F))
\geq \frac{|V(F)|}{\mu(F)}\geq \frac{|S|h}{h} = |S|.
\]
Thus \eqref{eq:topological-hall} holds for every $S\subseteq[m]$, and the
topological Hall criterion yields an independent transversal.  Taking
$r=k-1$ proves the final assertion.
If $d=0$, then $G$ is edgeless, so every choice of one vertex from each
nonempty part is an independent transversal.
\end{proof}

For fixed $k\geq 3$, the leading coefficient in
Corollary~\ref{cor:independent-transversal} is $2-2/k$, improving the
coefficient $2-1/(k-1)$ obtained from the direct connectivity argument in
\cite{AharoniAlonBerger}; consequently, the new threshold is smaller for all
sufficiently large $d$. The additive shift by one is unavoidable in the
corresponding spectral statement, as shown by the clique extensions in
Section~\ref{sec:sharpness}.

\section{Clique extensions and sharpness}
\label{sec:sharpness}

For a graph $G$ and a positive integer $t$, let $G[K_t]$ denote the graph
obtained by replacing every vertex $v$ of $F$ with a clique $V_v$ of order
$t$, and replacing every edge $uv$ of $F$ with all edges between $V_u$ and
$V_v$.

\begin{proposition}\label{prop:extension}
Let $k\geq 2$, and suppose that $\alpha (G[N_G(v)])\leq k-1$ for every $v\in V(G)$.
For every positive integer $t$, the graph $G[K_t]$ has the following properties.
\begin{enumerate}
\item[$(1)$] Every open neighborhood in $G[K_t]$ has independence number at most $k-1$.
\item[$(2)$] $\Delta(G[K_t]) + 1 = t(\Delta(G)+1)$.
\item[$(3)$] If $L(G)\bm{x} = \theta\bm{x}$, then the vector that is constant with value
$x_v$ on each clique $V_v$ is an eigenvector of $L(G[K_t])$ with eigenvalue $t\theta$.
\end{enumerate}
\end{proposition}

\begin{proof}
Fix $w\in V_v$. A vertex of $V_v\setminus\{w\}$ is adjacent to every other
vertex in $N_{G[K_t]}(w)$. Thus an independent set in this neighborhood that
meets $V_v\setminus\{w\}$ has order one.  An independent set disjoint from
$V_v$ contains at most one vertex from each $V_u$ with $u\in N_G(v)$, and the
corresponding vertices $u$ form an independent set in $G[N_G(v)]$. This proves (i).

Every vertex in $V_v$ has degree $(t-1) + t d_G(v) = t (d_G(v) + 1) - 1$,
which proves (ii).

For (iii), let $\widetilde{\bm{x}}$ have value $x_v$ on $V_v$. If $w\in V_v$, then
\begin{align*}
(L(G[K_t])\widetilde{\bm{x}})_w
& = (t d_G(v) + t - 1) x_v - (t-1) x_v - t\sum_{u\in N_G(v)} x_u \\
& = t\bigg(d_G(v) x_v - \sum_{u\in N_G(v)} x_u\bigg) 
= t(L(G) \bm{x})_v = t\theta x_v,
\end{align*}
This finish the proof of (iii).
\end{proof}

Proposition~\ref{prop:extension} also explains why the additive term in
Theorem~\ref{thm:main} is naturally attached to $\Delta(G)+1$. Indeed, if the
asymptotic estimate $\mu(G)\leq\big(\frac{2(k-1)}{k} + o(1)\big)\Delta(G)$
holds for all graphs with local independence number at most $k-1$, then applying
it to $G[K_t]$, using Proposition~\ref{prop:extension}, and letting
$t\to\infty$ gives
\[
\mu(G)\leq \frac{2(k-1)}{k} (\Delta(G) + 1).
\]
% Thus the asymptotic question of Aharoni, Alon, and Berger is equivalent to the
% exact inequality in Theorem~\ref{thm:main}.

We finish by verifying sharpness using the clique extensions employed in
\cite{AharoniAlonBerger}. Let $G$ be any $(k-1)$-regular bipartite graph
with bipartition $(X,Y)$, and let $H=G[K_t]$. By
Proposition~\ref{prop:extension}, every open neighborhood of $H$ has
independence number at most $k-1$, and $\Delta(H)=kt-1$.
Let $\bm{x}$ be $1$ on $X$ and $-1$ on $Y$. Since $G$ is $(k-1)$-regular and bipartite, $L(G)\bm{x} = 2(k-1)\bm{x}$.
Proposition~\ref{prop:extension} (iii) therefore gives
$\mu(H)\geq 2(k-1)t = \frac{2(k-1)}{k}\big(\Delta(H) + 1\big)$.
The reverse inequality follows from Theorem~\ref{thm:main}, so equality holds.
Taking, for example, $G=K_{k-1,k-1}$ produces an equality example for every $k\geq 2$, $t\geq 1$, and
$\frac{\mu(G[K_t])}{\Delta(G[K_t])}\to\frac{2(k-1)}{k}$
as $t\to\infty$.

\section*{Declaration on the use of AI}

The authors used generative AI tools to assist in discussing proof strategies, checking proofs, and
improving the exposition. The authors take full responsibility for the mathematical arguments, results,
and conclusions, all of which were carefully reviewed and verified by them.


\begin{thebibliography}{99}

\bibitem{AharoniAlonBerger}
R. Aharoni, N. Alon, E. Berger,
Eigenvalues of $K_{1,k}$-free graphs and the connectivity of their independence complexes,
\emph{J. Graph Theory} \textbf{83} (2016), 384--391.

\bibitem{AharoniBergerMeshulam}
R. Aharoni, E. Berger, R. Meshulam,
Eigenvalues and homology of flag complexes and vector representations of graphs,
\emph{Geom. Funct. Anal.} \textbf{15}\,(3) (2005), 555--566.

\bibitem{AndersonMorley}
W. N. Anderson, T. D. Morley,
Eigenvalues of the Laplacian of a graph,
\emph{Linear Multilinear Algebra} \textbf{18} (1985), 141--145.

\bibitem{Bansal-Harris2023} N. Bansal, D. G. Harris,
Some remarks on hypergraph matching and the F\"uredi--Kahn--Seymour conjecture,
\emph{Random Structures \& Algorithms} \textbf{62} (2023), 52--67.

\bibitem{BrouwerHaemers}
A. E. Brouwer and W. H. Haemers, \emph{Spectra of Graphs},
Universitext, Springer, New York, 2012.

\bibitem{Caro}
Y. Caro, \emph{New Results on the Independence Number},
Technical Report, Tel Aviv University, 1979.

\bibitem{CioabaElzingaGregory}
S. M. Cioab\u{a}, R. J. Elzinga, D. A. Gregory,
Some observations on the smallest adjacency eigenvalue of a graph,
\emph{Discuss. Math. Graph Theory} \textbf{40} (2020), 467--493.

\bibitem{KoolenCaoYang2011}
J. H. Koolen, M. Cao, Q.Q. Yang, Recent progress on graphs with fixed smallest
eigenvalue: A Survey, \emph{Graphs Combin.} \textbf{37} (2021), 1139--1178.

\bibitem{MeshulamClique}
R. Meshulam,
The clique complex and hypergraph matching,
\emph{Combinatorica} \textbf{21}\,(1) (2001), 89--94.

\bibitem{Turan}
P. Tur\'{a}n,
Eine Extremalaufgabe aus der Graphentheorie,
\emph{Mat. Fiz. Lapok} \textbf{48} (1941), 436--452.

\bibitem{Wei}
V. K. Wei,
\emph{A Lower Bound on the Stability Number of a Simple Graph},
Bell Laboratories Technical Memorandum 81-11217-9, Murray Hill, NJ, 1981.

\end{thebibliography}
\end{document}